\documentclass[11pt,reqno]{amsart}
\usepackage[margin=20mm]{geometry}
\usepackage{amssymb}
\usepackage{bbm}
\usepackage{dsfont}
\usepackage{amsfonts}
\usepackage{pifont}
\usepackage{bbding}
\usepackage{latexsym}
\usepackage{mathrsfs}
\usepackage{amsmath}
\allowdisplaybreaks[4]
\usepackage{amsthm}
\usepackage{amssymb}
\usepackage{graphicx}
\usepackage{hyperref}
\usepackage[capitalise]{cleveref}
\usepackage{fancyhdr}
\usepackage{color}

\newtheorem{theorem}{\indent Theorem}[section]

\newtheorem{assumption}{\indent Assumption}[section]
\newtheorem{lemma}{\indent Lemma}[section]
\newtheorem{remark}{\indent Remark}[section]

\numberwithin{equation}{section}

\newcommand{\R}{\mathbb{R}}
\newcommand{\s}{\mathbb{S}}

\newcommand{\al}{\alpha}

\newcommand{\de}{\delta}

\newcommand{\proofof}[1]
\allowdisplaybreaks

\begin{document}
	
\allowdisplaybreaks

\title[Stable limit theorem]{A stable limit theorem for SDEs driven by multiplicative $\alpha$-stable processes}


\author{Kun Yin}

\address{\textbf{Kun Yin:} School of Mathematical Sciences, Shanghai Jiao Tong University, 200240 Shanghai, PR China.}\email{epsilonyk@sjtu.edu.cn}




\allowdisplaybreaks

\begin{abstract}
We derive a stable limit theorem for stochastic differential equations driven by multiplicative $\alpha$-stable processes. A key ingredient is the $L^1$-exponential contractivity estimate for the SDEs. The limiting process is a non-degenerate symmetric $\alpha$-stable process with an averaged Lévy measure.

\end{abstract}

\maketitle

\section{Introduction}

\subsection{Background}
It is known that the following model represents a classical framework in statistical fluid mechanics,
$$
dY_t=f(Y_t)dt+\text{L\'evy noise},
$$
which describes the dynamics of (anomalous) diffusive particles convected by a random velocity field $f(y)$. The long-time behavior and limit theorems of SDEs driven by L\'evy processes have been extensively studied.

For additive $\alpha$-stable noise, exponential ergodicity follows from dissipative drifts \cite{MBM,SXX,JW}. In particular, \cite{JW2} derived $L^p$-Wasserstein contraction by a coupling method under a partially dissipative condition.

Functional limit theorems for additive functionals of Markov processes have been studied extensively; see, e.g., \cite{BWY,XC,KV,AS}. The ``martingale approximation'' and ``coupling'' approaches developed in \cite{JKO,PSV} have been widely used, for example, in \cite{GP,MI}.
It is also worth mentioning that the Stein method was employed in \cite{ADB} to obtain distributional approximations in functional approximations by the Wiener process and other Gaussian processes, and
was subsequently extended to stable processes in \cite{X2,X1}.

In particular, a positive recurrent one-dimensional diffusion process described by an It\^o equation was studied in \cite{LB}. Under assumptions
involving slowly varying functions and different scalings, the author established the convergence of rescaled additive functionals to Brownian motion and $\alpha$-stable processes. For functionals of stable-like
processes, \cite{BF} investigated processes with periodic coefficients
using Doeblin's classical results on invariant measures.

It is worth noting that additive L\'evy noise together with a dissipative drift makes it easier to obtain pathwise exponential contractivity, which in turn yields exponential ergodicity. In the multiplicative case,
however, the noise no longer cancels in the difference of two solutions, making the analysis more delicate.

\subsection{Sketch of this paper}

We consider the following stochastic differential equation driven by a multiplicative $\alpha$-stable process, where  $L_{t}$ is an isotropic  $\alpha$-stable process with $1<\alpha<2$,
\begin{align}\label{multca}
	dY_{t}=f(Y_{t})dt+\delta(Y_{t-})dL_{t},\quad
	Y_{0}=y\in \mathbb{R}^{d },
\end{align}	
we study the stable limit theorem in Theorem \ref{addi-func},  	$$
	(A_t^\varepsilon)_{t\geq0}:=\bigg(\varepsilon^{1/\alpha}
\int_0^{t/\varepsilon}f(Y_s)\,ds\bigg)_{t\geq0} \xrightarrow{\varepsilon\to0}  (A_t)_{t\geq0}, \quad   \text{ in finite-dimensional distributions,}
$$
where $A_t$ is a non-degenerate symmetric $\alpha$-stable process, whose L\'evy measure is given in \eqref{lim-lev}. The key ingredient is the $L^1$-exponential contractivity of $Y_t$ established in Lemma \ref{str-contra}.

\begin{remark}
	It is worth emphasizing that the $L^1$-exponential contractivity	obtained in this paper is more involved than in the case of additive Lévy noise. In general, non-synchronous couplings, such as reflection coupling, may yield contraction in expectation or in Wasserstein distance, but do not	provide the contraction of the trajectories \cite{JW2}.
	Under the synchronous coupling, the two solutions are driven by the
	same Lévy process. For the multiplicative noise in \eqref{multca},
	however, we have
	\[
	d\big(Y_t^{y_1}-Y_t^{y_2}\big)=\big(f(Y_t^{y_1})-f(Y_t^{y_2})\big)dt+
	\big(\delta(Y_{t-}^{y_1})-\delta(Y_{t-}^{y_2})\big)dL_t,
	\]
unlike the additive-noise case, the difference process still contains a nontrivial noise term. Instead of relying on cancellation, we control this term using the dissipativity of $f$ and the Lipschitz continuity of $\delta$, which yields the $L^1$-exponential contractivity in \eqref{strong-est}.
\end{remark}

\subsection{Organization of this paper}
This paper is organized as follows. In Section \ref{setting}, we introduce the notation, assumptions, and state the main results concerning stable limit theorem of the SDE. In Section \ref{con} we prove the $L^1$-exponential contractivity of $Y_t$. In Section \ref{add} we prove the stable limit theorem. Appendix \ref{tech} is devoted to the technical estimates of nonlocal operator, and moment estimates of rescaled processes.

\section{Some settings and main results}\label{setting}

\subsection{Notations and assumptions}
Let $\mathbb R^d$ be Euclidean space, $\langle\cdot,\cdot\rangle$ and $|\cdot|$ inner product and norm, $\|\cdot\|$ the matrix  operator norm.  Let $(\Omega,\mathcal F,(\mathcal F_t)_{t\ge0},\mathbb P)$ be a probability space with expectation $\mathbb E$.  $\xrightarrow{\mathcal D}$, $\overset{\mathcal D}{=}$ denote convergence in distribution, equality in law respectively. $\mathcal{B}(\mathbb{R}^d\setminus\{0\})$ denotes the Borel $\sigma$-algebra. $P_t$ denotes the transition semigroup of $Y_t$.  $Y^{y}_{t}$  denotes the process $Y_{t}$ starting from $y$. 

$ C_{b}(\mathbb{R}^{d} )$=\{$u:\mathbb{R}^{d}\to \mathbb{R}$: $u$ is bounded and continuous\}. For $k\in \mathbb{N_+}$, we denote:

$ C^{k}(\mathbb{R}^{d} )$=\{$u:\mathbb{R}^{d}\to \mathbb{R}$: $u$ and all its partial derivatives up to order $k$ are continuous\}, 

$ C^{k}_{b}(\mathbb{R}^{d} )$=\{$u\in  C^{k}(\mathbb{R}^{d})$: all the partial derivatives of  orders $1,\cdots,k$ are bounded\}.

Define  
the infinitesimal generator  of $Y_t$  as
\begin{align}\label{l2}
	 \mathcal{L}u(y)
	 &=\int_{\mathbb R^{d}\setminus\{0\}}\big(u(y+\delta(y)z)-u(y)-\langle\delta(y)z,\nabla_{y}u(y)\rangle I_{\{|z|\leq1\}}\big)\nu(dz)\nonumber\\
	 &\quad+\left\langle f(y),\nabla_{y}u(y)\right\rangle ,
	 \end{align}
where $\nu(dz)=\frac{c_{\alpha,d}}{|z|^{d+\alpha}}dz$ is a symmetric L\'{e}vy measure, $c_{\alpha,d}>0$ is constant. The compensated Poisson measure of $\nu$ is defined as 
$ \tilde N(t,B)=N(t,B)-t\nu(B),$ $B\in\mathcal{B}(\mathbb{R}^d\setminus\{0\})$.

We next impose some important assumptions.

\begin{assumption}[Dissipative condition]\label{dis}
 There exists $\lambda>0$ such that for all
 $y_1,y_2\in\mathbb{R}^{d}$,
\begin{equation}\label{2.2}
|f(0)|<\infty,  \quad	\langle f(y_1)-f(y_2),y_1-y_2\rangle
\leq -\lambda |y_1-y_2|^2 .
\end{equation}		
\end{assumption}

\begin{assumption}[Lipschitz condition]\label{lip}
There exist $L_b,L_\delta>0$ such that for all
$y_1,y_2\in\mathbb{R}^{d}$,
	\begin{align*}
		|f(y_{1})-f(y_{2})| \leq L_b|y_{1}-y_{2}|,
		\quad	\|\delta(y_1)-\delta(y_2)\|
		\leq L_\delta |y_1-y_2|.
	\end{align*}
\end{assumption}

\begin{assumption}[Growth and boundedness conditions]\label{gro-bou}
There exist $C_f,C_\delta,c_\delta>0$ s.t.  for all
$y\in\mathbb{R}^{d}$,
\begin{align*}
 |f(y)| \leq C_{f}(1+|y|),\quad   
 c_\delta\le|\de(y) \hat{z}|\le C_\delta \ \text{for any}\ \hat{z}=z/|z|\in\s^{d-1}.
\end{align*}
\end{assumption}

\subsection{Main results}
  Next we state the main results of this paper.

\begin{theorem}\label{addi-func}
	
	Let  Assumptions $\ref{dis}$--$\ref{gro-bou}$ hold, with  $\lambda$ sufficiently large,
	then \eqref{multca} has a unique invariant measure  $\mu$. If $\bar{f}=\int_{\mathbb{R}^{ d}}f(y)\mu(dy)=0,$   $\de(y)\in C_{b}^{ 1}(\mathbb{R}^{ d})$,  then 
	$$
	(A_t^\varepsilon)_{t\geq0}:=\bigg(\varepsilon^{1/\alpha}
	\int_0^{t/\varepsilon}f(Y_s)\,ds\bigg)_{t\geq0} \xrightarrow{\varepsilon\to0}  (A_t)_{t\geq0}, \quad   \text{ in finite-dimensional distributions,}
	$$
	and $A_t$ is a non-degenerate symmetric $\alpha$-stable process with  L\'evy measure $\bar{\nu}$, which is defined as
	\begin{equation}\label{lim-lev}
		\bar{\nu}(B)=\int_{\mathbb R^d}
		\bigl(\nu\circ F_y^{-1}\bigr)(B)\,\mu(dy),
		\qquad \text{for any } B\in\mathcal{B}(\mathbb{R}^d\setminus\{0\}),
	\end{equation}
	where $\nu\circ F_y^{-1}$ is the image measure of $\nu$ with respect to $ F_y(z)=\delta(y)z$.
\end{theorem}

\section{$L^1$-exponential contractivity of \eqref{multca}}\label{con}

For $\eta>0$, we define
\begin{align}\label{vz}
V_\eta(z)=\sqrt{|z|^2+\eta^2},
\end{align}
then
\begin{align}\label{grad}
	|\nabla V_\eta(z)|
	=\bigg|\frac{z}{\sqrt{|z|^2+\eta^2}}\bigg|\leq1,\quad  \quad
	\|\nabla^2V_\eta(z)\|
	=\bigg\|\frac{1}{\sqrt{|z|^2+\eta^2}}I_{d}
	-\frac{zz^{\top}}
	{(|z|^2+\eta^2)^{3/2}}\bigg\|\leq
	\frac{1}{\sqrt{|z|^2+\eta^2}}, 
\end{align}
indeed, from $|\nabla V_\eta(z)|\leq1$, we can deduce that $V_\eta$ is globally $1$-Lipschitz continuous,
\begin{align}\label{v-lip}
	|V_\eta(y)-V_\eta(z)|\leq|y-z|.
\end{align}

\begin{lemma}[$L^1$-exponential contractivity of $Y_t$]	\label{str-contra}
	Let  Assumptions $\ref{dis}$--$\ref{gro-bou}$ hold,    
	with  $\lambda$ sufficiently large,
	then for any	$y_1,y_2\in\mathbb R^{d}$, there exists $\beta>0$ s.t.
	\begin{equation}\label{strong-est}
		\mathbb E|Y_t^{y_1}-Y_t^{y_2}|
		\leq	e^{-\beta t}|y_1-y_2|,
		\qquad t\geq0,
	\end{equation}
	then there exists a unique invariant measure $\mu$ of $\eqref{multca}$, and for any $g\in C_b^1$, there exists $C>0$,
	\begin{equation}\label{L-gen-exp}
		\left| P_{t}g(y)-\bar{g}\right| \leq C\cdot  \|\nabla g\|_\infty  e^{-\beta t}(1+|y|),\qquad t\geq0, \ y\in \mathbb R^{d}.
	\end{equation}
\end{lemma}
\begin{proof}
	Let $y_1,$ $y_2$ be any given initial values, define
	$$	Z_t=Y_t^{y_1}-Y_t^{y_2}, \quad Z_0=y_1-y_2,$$
	then
	$$
	dZ_t=F_tdt+G_tdL_t,
	$$
	where 
	\begin{align*}
		F_t=f(Y_t^{y_1})-f(Y_t^{y_2}),\quad  G_t=\delta(Y_{t-}^{y_1})-\delta(Y_{t-}^{y_2}),	
	\end{align*}
if $Z_t=0$ then $F_t=G_t=0$ and the two solutions remain identical thereafter,
the desired estimate is immediate,  thus we apply
Lemma \ref{v-A} on $\{Z_t\neq0\}$.

	By Assumption \ref{lip}, we have
	\begin{equation}\label{G-est}
		\|G_t\|\leq L_\delta |Z_{t-}|.
	\end{equation}

Applying the It\^o's formula to
	$V_\eta(Z_t)$ and taking expectation, we obtain
	\begin{align*}
		\frac{d}{dt}\mathbb EV_\eta(Z_t)
		&=\mathbb E
		\left\langle\nabla V_\eta(Z_t),F_t
		\right\rangle+\mathbb E
		\int_{\mathbb R^{d}\setminus\{0\}}
		\Big(V_\eta(Z_{t-}+G_th)-V_\eta(Z_{t-})
		-\langle\nabla V_\eta(Z_{t-}),G_th\rangle I_{\{|h|\leq1\}}\Big)\nu(dh)\\
		&=I_1^\eta(t)+I_2^\eta(t).
	\end{align*}

	For $I_1^\eta(t)$,  by \eqref{grad} 
 and Assumption \ref{dis} we get
	\begin{align}\label{i1}
	I_1^\eta(t)
		\leq 	-\lambda
		\mathbb E
		\frac{|Z_t|^2}
		{\sqrt{|Z_t|^2+\eta^2}}.
	\end{align}
	
	Next, we estimate $I_2^\eta(t)$.  If $\|G_t\|=0$, then $I_2^\eta(t)=0$. For $\{\|G_t\|>0\}$, we define
	\[
	K_t=\frac{G_t}{\|G_t\|},\qquad G_t=\|G_t\|K_t,
	\]
	let $w=\|G_t\| h$, then $$G_t h=\|G_t\| K_th=\|G_t\| K_t\frac{w}{\|G_t\|}=K_tw,$$
	by the scaling property of $\nu$, 
	\begin{align*}
		\nu(dh)
		=c_{\al,d} |h|^{-d-\alpha}dh
		=c_{\al,d} \left( \frac{|w|}{\|G_t\|} \right)^{-d-\alpha} \|G_t\|^{-d}dw 
		&= c_{\al,d} |w|^{-d-\alpha} \|G_t\|^{d+\alpha} \|G_t\|^{-d}dw = \|G_t\|^{\alpha}\,\nu(dw),
	\end{align*}
	then from \eqref{z1-2-A} in Lemma \ref{v-A}, and $\|K_t\|\leq1,$
	\begin{align*}
		&\int_{\mathbb R^{d}\setminus\{0\}}
		\Big(V_\eta(z+\|G_t\|\frac{G_t}{\|G_t\|}h)-V_\eta(z)
		-\langle\nabla V_\eta(z),\|G_t\|\frac{G_t}{\|G_t\|}h\rangle
		I_{\{|h|\leq1\}}	\Big)\nu(dh)\\
		&=\|G_t\|^{\alpha}
		\int_{\mathbb R^{d}\setminus\{0\}}
		\Big(V_\eta(z+K_tw)-V_\eta(z)-\langle\nabla V_\eta(z),K_tw\rangle I_{\{|w|\leq\|G_t\|\}}\Big)\nu(dw)\\
		&\leq	C_{\alpha,d}\|G_t\|^{\alpha}|z|^{1-\alpha},
	\end{align*}
	therefore,
\begin{align*}
I_2^\eta(t)	&=	\mathbb E\int_{\mathbb R^{d}\setminus\{0\}}
\Big(V_\eta(Z_{t-}+G_th)-V_\eta(Z_{t-})-\langle\nabla V_\eta(Z_{t-}),G_th\rangle I_{\{|h|\leq1\}}\Big)\nu(dh)\\
&=\mathbb E\bigg[\|G_t\|^{\alpha}\int_{\mathbb R^{d}\setminus\{0\}}\Big(V_\eta(Z_{t-}+K_tw)-V_\eta(Z_{t-})
-\langle\nabla V_\eta(Z_{t-}),K_tw\rangle
I_{\{|w|\leq\|G_t\|\}}\Big)\nu(dw)\bigg]\\
&\leq C_{\alpha,d}\mathbb E\left(\|G_t\|^{\alpha}	|Z_{t-}|^{1-\alpha}	\right),
\end{align*}
from \eqref{G-est}, we have
$
\|G_t\|^{\alpha}|Z_{t-}|^{1-\alpha}\leq L_\delta^{\alpha}
	|Z_{t-}|^{\alpha}|Z_{t-}|^{1-\alpha}=L_\delta^{\alpha}|Z_{t-}|,
$
since $(Z_t)_{t\geq0}$ is càdlàg, $Z_{t-}=Z_t$ for $dt\otimes d\mathbb P$-a.e. $t$,	thus,
\begin{align}\label{i2}
I_2^\eta(t)\leq C_{\alpha,d}L_\delta^{\alpha}\mathbb E|Z_{t}|.
\end{align}

	Combining \eqref{i1} and \eqref{i2}, we obtain
	\begin{equation}\label{d-v}
		\frac{d}{dt}\mathbb E V_\eta(Z_t)
		\leq-\lambda \mathbb E\bigg[\frac{|Z_t|^2}{\sqrt{|Z_t|^2+\eta^2}}\bigg]
		+C_{\alpha,d}L_\delta^{\alpha}\mathbb E|Z_t|,
	\end{equation}
	integrating \eqref{d-v} from $0$ to $t$, we obtain
	\begin{align}\label{inte-v}
	\mathbb E V_\eta(Z_t)- V_\eta(Z_0)\leq-\lambda\int_0^t \mathbb E\bigg[\frac{|Z_s|^2}{\sqrt{|Z_s|^2+\eta^2}}\bigg]ds+
	C_{\alpha,d}L_\delta^{\alpha}\int_0^t\mathbb E|Z_s|\,ds,
	\end{align}
by the monotone convergence theorem, passing $\eta$ to the limit in \eqref{inte-v}, we yield
\begin{align*}
\mathbb E|Z_t|	\leq |Z_0|-\left(	\lambda-C_{\alpha,d}L_\delta^{\alpha}	\right)\int_{0}^{t}	\mathbb E|Z_s|ds,
\end{align*}
	by Gronwall's inequality, there exists $\beta>0$ such that
$
	\mathbb E|Z_t|\leq e^{-\beta t}|Z_0|,
$
	which means that
	\begin{align*}
	\mathbb E|Y_t^{y_1}-Y_t^{y_2}|\leq e^{-\beta t}|y_1-y_2|.
	\end{align*}

For any bounded measurable function $g:\mathbb{R}^{d}\rightarrow \mathbb{R}$,  we define the semigroup
$ P_{t}g(y)=\mathbb{E}g(Y_{t}^{ y})$,  $t\geq 0,$  $y\in \mathbb{R}^{d}.$ By \eqref{3.6} and the Feller property, the Krylov--Bogoliubov theorem
gives the existence of an invariant measure $\mu$, while
\eqref{strong-est} yields its uniqueness.

Define the average as
$
\bar{g}=\mu(g)=\int_{\R^{d}}g(y)\mu(dy).
$
	 From \eqref{strong-est}, for $g\in C^1_b$, $\exists C>0$,  for any $ t\geq0, \ y\in \mathbb R^{d},$
\begin{align*}
	\left| P_{t}g(y)-\mu(g)\right| 
	\leq \bigg|  \int_{\R^{d}}\big(\mathbb{E}g(Y^{y}_t)-\mathbb{E}g(Y^{z}_t) \big) \mu(dz)\bigg|&\leq C \|\nabla g\|_\infty\int_{\R^{d}}	\mathbb E|Y_t^{y}-Y_t^{z}| \mu(dz)\\
	& \leq C\|\nabla g\|_\infty e^{-\beta t}(1+|y|).
\end{align*}
\end{proof}

\section{Proof of Theorem \ref{addi-func}}\label{add}

In this section we prove Theorem \ref{addi-func}, the method mainly refers to \cite[p.1094, Theorem 1]{BF}.

\begin{proof}
By the symmetry of $\nu$, the Lévy-Itô decomposition gives
\begin{equation*}
	Y_t-y=\int_0^t f(Y_s)\,ds+\int_0^t\delta(Y_{s-})\,dL_s=\int_0^t f(Y_s)\,ds+\int_0^t\int_{\mathbb R^d\setminus\{0\}}
	\delta(Y_{s-})z\,\tilde N(ds,dz),
\end{equation*}
then we have
\begin{equation}\label{a-m}
A_t^\varepsilon=\varepsilon^{1/\alpha}
\int_0^{t/\varepsilon}f(Y_s)\,ds=R_t^\varepsilon-M_t^\varepsilon,
\end{equation}
where
\begin{align*}
R_t^\varepsilon:=\varepsilon^{1/\alpha}\bigl(Y_{t/\varepsilon}-y\bigr),\quad	M_t^\varepsilon:=\varepsilon^{1/\alpha}\int_0^{t/\varepsilon}\int_{\mathbb R^d\setminus\{0\}}
\delta(Y_{s-})z\,\tilde N(ds,dz).
\end{align*}

We first deal with $R^\varepsilon_t$. By the moment estimate
\eqref{3.6}, with $m=1$ we obtain,
\begin{align}\label{R-est}
	\sup_{t\geq0}\mathbb E|R_t^\varepsilon|
	\leq\varepsilon^{1/\alpha}	\big(
	\sup_{t\geq0}\mathbb E|Y_{t/\varepsilon}|+|y|
	\big)=\varepsilon^{1/\alpha}
	\big(\sup_{t\geq0}\mathbb E|Y_t^\varepsilon|+|y|\big)
	\leq C\varepsilon^{1/\alpha}(1+|y|),
\end{align}
hence, for any fixed $t\ge0$,
 $R_t^\varepsilon\to0$ in $L^1$. Therefore, for any finite
collection $0\le t_1<\cdots<t_k\le T$,
\begin{align}\label{rt}
(R_{t_1}^\varepsilon,\ldots,R_{t_k}^\varepsilon)
\xrightarrow{\mathbb{P}}0,  \quad \text{as}\quad \varepsilon\to 0.
\end{align}

By the semimartingale
convergence theorem \cite[p.459, Theorem 1.18]{JS}, we identify the weak limit of $M^\varepsilon_t$ by verifying the convergence of predictable characteristics. It is sufficient to prove that they converge to those of an $\alpha$-stable process $A_t$
\cite[p.110]{ET}. Our proof mainly follows
\cite[p.1094, Theorem 1]{BF}.

We observe that the process $M^\varepsilon_{t}$ is a purely discontinuous local martingale. Then for the sequence $M^\varepsilon_{t}$, we denote by $B^\varepsilon_{t}$ the first
characteristic (the predictable finite variation drift), by
$\widetilde C^\varepsilon_{t}$ the modified second characteristic (the
predictable process collecting quadratic variation of the truncated small jumps), and by $\nu^\varepsilon$ the third characteristic (the
compensator of jump measure), we next
prove the following aspects:
\begin{enumerate}
	\item $B_{t}^\varepsilon=0$;    
	\item $\widetilde C^\varepsilon_t\to\widetilde C_t$ in probability for all $t\ge0$;
	\item $(g*\nu^\varepsilon)_t\to(g*\bar\nu)_t$ in probability for all $t\ge0$
	and $g\in C_b(\mathbb R^d)$ vanishing in a neighborhood of $0$.
\end{enumerate}

	Choose a smooth truncation function $\eta\in C_c^\infty(\mathbb R^d)$
	such that
	$$
	\eta(x)=x \ \text{for}\ |x|\le \frac12,
	\qquad	\eta(x)=0 \ \text{for}\ |x|\ge 1,	\qquad
	\text{and}\ \eta(-x)=-\eta(x).
	$$
	
	
	We split the following argument into $4$ steps.
	
	\textbf{Step 1:} For the first characteristic $B_t^\varepsilon$,  for fixed $y\in\mathbb R^d$, the function
	$
	z\mapsto \eta(\delta(y)z)-\delta(y)z
	$
	is integrable with respect to $\nu$ and is odd. As $\nu$ is
	symmetric, $\int_{\mathbb R^d\setminus\{0\}}
	(\eta(\delta(y)z)-\delta(y)z)\nu(dz)=0 $, then
	\begin{equation}\label{bv-b}
	B_t^\varepsilon=
	\int_0^t\int_{\mathbb R^d\setminus\{0\}}
	\big(\eta(\delta(Y_{s/\varepsilon-})z)
	-\delta(Y_{s/\varepsilon-})z\big)
	\nu(dz)\,ds=0.
		\end{equation}
The modified second characteristic collecting the quadratic variation of the truncated small jumps is
	\begin{align}	\label{sec-char}
	\widetilde C^\varepsilon_t
	&=\int_0^t	\int_{\mathbb R^d\setminus\{0\}}	\eta(\delta(Y_{s/\varepsilon-})z)
	\eta(\delta(Y_{s/\varepsilon-})z)^\top\nu(dz)\,ds.
	\end{align}
Let $\mathcal G_t^\varepsilon:=\mathcal F_{t/\varepsilon}$,
 the integral of $g\in C_b(\mathbb R^d)$ vanishing in a neighborhood of $0$ with respect to the third characteristic $\nu^\varepsilon$ is given by,  
\begin{align}\label{thir-char}
	(g*\nu^\varepsilon)_t=\int_0^t\int_{\mathbb R^d\setminus\{0\}}
	g(z)\,\nu^\varepsilon(ds,dz)=
	\int_0^t\int_{\mathbb R^d\setminus\{0\}}
	g(\delta(Y_{s/\varepsilon-})z)\,\nu(dz)\,ds.
\end{align}

\textbf{Step 2:} For the modified second characteristic $\widetilde C^\varepsilon_t$ defined in \eqref{sec-char}, we define
$$
\Phi(y)=\int_{\mathbb R^d\setminus\{0\}}
\eta(\delta(y)z)\,\eta(\delta(y)z)^\top\,\nu(dz),
$$
since $\delta\in C_b^1(\mathbb R^d)$ and $\eta\in C_c^\infty(\mathbb R^d)$, we have $\Phi\in C_b^1(\mathbb R^d)$. Choose a component
$\Phi_{ij}$ and set
$$
p(y):=\Phi_{ij}(y)-\mu(\Phi_{ij}),
$$
then $p\in C_b^1$, $\mu(p)=0$, and by  \eqref{L-gen-exp},
$$
|P_s p(y)|\le C e^{-\beta s}(1+|y|),
$$
by the Markov property,  for $r\le s$, we have $	\mathbb E[p(Y_{s-})p(Y_{r-})]
=\mathbb E[p(Y_{r-})(P_{s-r}p)(Y_{r-})]$, thus 
\begin{align*}
	|\mathbb E[p(Y_{s-})p(Y_{r-})]|
	\leqslant\mathbb E[|p(Y_{r-})||(P_{s-r}p)(Y_{r-})|]
	&\le C e^{-\beta(s-r)}\,\mathbb E[|p(Y_{r-})|(1+|Y_{r-}|)]\le Ce^{-\beta(s-r)}(1+|y|),
\end{align*}
where the last inequality follows from the boundedness of $p$ and the moment estimate \eqref{3.6},
therefore,
\begin{align}\label{mod-2}
	\mathbb E\bigg[
	\bigg(\varepsilon\int_0^{t/\varepsilon}p(Y_{s-})\,ds
	\bigg)^2\bigg]
	=2\varepsilon^2	\int_0^{t/\varepsilon}\int_0^{s}
	\mathbb E[p(Y_{s-})p(Y_{r-})]\,dr\,ds
	&\le2C\varepsilon^2(1+|y|)	\int_0^{t/\varepsilon}\int_0^{s}
	e^{-\beta(s-r)}\,dr\,ds\nonumber\\
	&\le\frac{2Ct}{\beta\varepsilon}\varepsilon^2(1+|y|)
	=C(1+|y|)\varepsilon,
\end{align}
which means that 
$$
\varepsilon\int_0^{t/\varepsilon}p(Y_{s-})\,ds
\xrightarrow{L^2}0  \quad \text{as}\quad \varepsilon\to 0,
$$
and consequently
$$
\varepsilon\int_0^{t/\varepsilon}\Phi_{ij}(Y_{s-})\,ds
\xrightarrow{L^2}t\,\mu(\Phi_{ij}),
$$
since this holds for any component, we get
$$
\varepsilon\int_0^{t/\varepsilon}\Phi(Y_{s-})\,ds
\xrightarrow{L^2}t\int_{\mathbb R^d}\Phi(y)\,\mu(dy).
$$
Finally, by \eqref{sec-char} and the change of variables
$r=s/\varepsilon$,
$$
\widetilde C^\varepsilon_t=\int_0^t
\int_{\mathbb R^d\setminus\{0\}}\eta(\delta(Y_{s/\varepsilon-})z)
\eta(\delta(Y_{s/\varepsilon-})z)^\top
\nu(dz)\,ds=\varepsilon\int_0^{t/\varepsilon}\Phi(Y_{r-})\,dr,
$$
therefore,
\begin{align}\label{cv-c}
\widetilde C^\varepsilon_t
\xrightarrow{L^2}
\widetilde C_t=t\int_{\mathbb R^d}\Phi(y)\,\mu(dy)=t\int_{\mathbb R^d}
\int_{\mathbb R^d\setminus\{0\}}\eta(\delta(y)z)\eta(\delta(y)z)^\top\nu(dz)\mu(dy)\quad \text{as}\quad \varepsilon\to 0.
\end{align}

\textbf{Step 3:}
For the third characteristic $(g*\nu^\varepsilon)_t$ given in
\eqref{thir-char}, we need to prove that
\[
(g*\nu^\varepsilon)_t
\xrightarrow{\mathbb P}
(g*\bar\nu)_t\quad \text{as}\quad \varepsilon\to 0,
\]
for any $g\in C_b(\mathbb R^d)$ vanishing in a neighborhood of  $0$.
Since \eqref{L-gen-exp} applies to $C_b^1(\mathbb R^d)$ functions, we consider
\[
g\in C_c^\infty(\mathbb R^d),\qquad
g=0\ \text{in a neighborhood of}\ 0,
\]
then there exists $\gamma>0$ such that
$
g(z)=0$ for $|z|\leq\gamma,
$ define
\[
\Psi(y):=
\int_{\mathbb R^d\setminus\{0\}}
g(\delta(y)z)\,\nu(dz),
\]
then $ |\Psi(y)|\leq
\|g\|_\infty
\nu(\{|z|>\gamma/C_\delta\})
< \infty$, and $\|\nabla\Psi\|_\infty
\leq
\|\nabla g\|_\infty\|\nabla \delta\|_\infty
\int_{\{|z|>\gamma/C_\delta\}}|z|\,\nu(dz)
<\infty,$ 
thus
$
\Psi\in C_b^1(\mathbb R^d).
$
Let
$
p(y):=\Psi(y)-\mu(\Psi),
$
then $p\in C_b^1(\mathbb R^d)$, $\mu(p)=0$, by
\eqref{L-gen-exp},
\[
|P_s p(y)|\le Ce^{-\beta s}(1+|y|),
\]
repeating the argument leading to \eqref{mod-2}, we obtain
\[
\mathbb E\bigg[
\bigg(\varepsilon\int_0^{t/\varepsilon}p(Y_{s-})\,ds
\bigg)^2\bigg]\le C\varepsilon\longrightarrow0\quad \text{as}\quad \varepsilon\to 0,
\]
thus,
\[
\varepsilon\int_0^{t/\varepsilon}\Psi(Y_{s-})\,ds
\xrightarrow{L^2}t\,\mu(\Psi).
\]
By \eqref{thir-char} and the change of variables $r=s/\varepsilon$,
\[
(g*\nu^\varepsilon)_t=\varepsilon\int_0^{t/\varepsilon}\Psi(Y_{r-})\,dr,
\]
hence for $g\in C_c^\infty(\mathbb R^d)$ vanishing 
in a neighborhood of $0$,    as $\varepsilon\to 0$,
\begin{align}\label{thir}
	(g*\nu^\varepsilon)_t
	\xrightarrow{L^2}
	t\mu(\Psi)=t\int_{\mathbb R^d}	\int_{\mathbb R^d\setminus\{0\}}	g(\delta(y)z)\nu(dz)\mu(dy)&=t\int_{\mathbb R^d\setminus\{0\}}	g(z) \int_{\mathbb R^d}
	\left(\nu\circ F_y^{-1}\right)(dz)\mu(dy)\nonumber\\
	&=t\int_{\mathbb R^d\setminus\{0\}}	g(z) \bar \nu(dz)=(g*\bar \nu)_t,
\end{align}
where $\bar \nu(dz)=\int_{\mathbb R^d}
\left(\nu\circ F_y^{-1}\right)(dz)\mu(dy) $  and $F_y(z)=\delta(y)z$.

Next, for any
$g\in C_b(\mathbb R^d)$ vanishing in a neighborhood of  $0$, let
$r_0>0$ be such that
\[
g(z)=0,\qquad |z|\le r_0.
\]
For $R>1$, choose cut-off function $\chi_R\in C_c^\infty(\mathbb R^d)$ satisfying
\[
\chi_R(z)=1\ \text{for}\ |z|\le R,\qquad
\chi_R(z)=0\ \text{for}\ |z|\ge 2R,
\]
and set
\[
g_R(z):=g(z)\chi_R(z),
\]
then $g_R$ has compact support and $g_R=0$ on $B_{r_0}(0)$.
Let $\rho^\kappa$ be a standard mollifier  with
$\operatorname{supp}\rho^\kappa\subset B_\kappa(0)$ and define 
\[
g_R^\kappa:=g_R\ast\rho^\kappa, 
\]
for $\kappa<r_0/2$, we have
\[
g_R^\kappa=0\quad\text{on}\quad  B_{r_0/2}(0),
\]
and for each fixed $R$,
\[
g_R^\kappa\in C_c^\infty(\mathbb R^d),
\qquad\|g_R^\kappa-g_R\|_\infty\longrightarrow0\quad \text{as}\quad \kappa\to0.
\]
In particular, we let $\kappa=\kappa(R)<r_0/2$ sufficiently small such that  $\kappa(R)\to0 $ as $R\to \infty$, and
\[
\|g_R^{\kappa(R)}-g_R\|_\infty\le R^{-1},
\]
by \eqref{thir-char}, there exists $C>0$ s.t.
\begin{align}\label{kapp}
	\left|	(g_R^{\kappa(R)}*\nu^\varepsilon)_t	-	(g_R*\nu^\varepsilon)_t	\right|	\le t\|g_R^{\kappa(R)}-g_R\|_\infty	\int_{\{|z|\ge r_0/(2C_\delta)\}}\nu(dz)	\le C tR^{-1},
\end{align}
on the other hand, since $g_R(z)=g(z)$ for $\{|z|\le R\}$, 
\begin{align}\label{r}
	\left|	(g_R*\nu^\varepsilon)_t	-(g*\nu^\varepsilon)_t
	\right|	\le	t\|g\|_\infty\int_{\{|z|>R/C_\delta\}}\nu(dz)	\le CtR^{-\alpha},
\end{align}
here we used the fact that
$
\nu(\{|z|>R\})\le CR^{-\alpha}.
$
Combining \eqref{kapp} and \eqref{r}, we obtain
\begin{align}\label{k-a}
\left|
(g_R^{\kappa(R)}*\nu^\varepsilon)_t-(g*\nu^\varepsilon)_t
\right|\longrightarrow0\quad \text{as} \quad R\to\infty,
\end{align}
uniformly in $\varepsilon$.
Similarly, since
$
\bar\nu(\{|z|>R\})
\le\int_{\mathbb R^d}
\nu(\{|z|>R/C_\delta\})\,\mu(dy)
\le CR^{-\alpha},
$
we have
\begin{align*}
	\bigg|	\int_{\mathbb R^d\setminus\{0\}} g_R^{\kappa(R)}(z)\,\bar\nu(dz)	-	\int_{\mathbb R^d\setminus\{0\}} g_R(z)\,\bar\nu(dz)	\bigg|	&\le
	R^{-1}\,\bar\nu(\{|z|\ge r_0/2\}),
\end{align*}
and
\begin{align*}
	\bigg|	\int_{\mathbb R^d\setminus\{0\}} g_R(z)\,\bar\nu(dz)	-\int_{\mathbb R^d\setminus\{0\}} g(z)\,\bar\nu(dz)
	\bigg|	\le	\|g\|_\infty\bar\nu(\{|z|>R\})
	\le CR^{-\alpha},
\end{align*}
thus,
\begin{align}\label{r-r-2}
	\bigg|\int_{\mathbb R^d\setminus\{0\}}
g_R^{\kappa(R)}(z)\,\bar\nu(dz)-\int_{\mathbb R^d\setminus\{0\}}
g(z)\,\bar\nu(dz)	\bigg|
\longrightarrow 0
 \quad \text{as}\quad R\to\infty.
\end{align}

Since $g_R^{\kappa(R)}\in
C_c^\infty(\mathbb R^d\setminus\{0\})$, \eqref{thir} applies to
$g_R^{\kappa(R)}$.  Combining \eqref{thir}, \eqref{k-a} and \eqref{r-r-2},
letting $\varepsilon\to 0$ and $R\to\infty$, for any
$g\in C_b(\mathbb R^d)$ vanishing in a neighborhood of  $0$, we obtain 
\begin{align}\label{gv-g}
	(g*\nu^\varepsilon)_t
	\xrightarrow{\mathbb P}
	(g*\bar\nu)_t=t\int_{\mathbb R^d\setminus\{0\}}
	g(z)\,\bar\nu(dz) \quad \text{as}\quad \varepsilon\to 0,
\end{align}
where $\bar{\nu}$ is a symmetric L\'evy measure, defined by
\begin{align}\label{A-mea}
	\bar\nu(B):=\int_{\mathbb R^d}
	\left(\nu\circ F_y^{-1}\right)(B)\,\mu(dy),\quad F_y(z)=\delta(y)z,\quad B\in\mathcal{B}(\mathbb{R}^d\setminus\{0\}),
\end{align}
 the non-degeneracy of $\bar{\nu}$ follows from Assumption \ref{gro-bou}, 
 and $$\bar{\nu}(-B)
=\int_{\mathbb{R}^d}(\nu\circ F_y^{-1})(-B)\,\mu(dy)
=\int_{\mathbb{R}^d}(\nu\circ F_y^{-1})(B)\,\mu(dy)
=\bar{\nu}(B),$$ then for any $a>0$, $\bar\nu(aB)=a^{-\alpha}\bar\nu(B),$  and 
$$
\int_{\mathbb{R}^d\setminus\{0\}}(1\wedge |z|^2)\bar{\nu}(dz)
=\int_{\mathbb{R}^d}\int_{\mathbb{R}^d\setminus\{0\}}
(1\wedge|\delta(y)z|^2)\nu(dz)\,\mu(dy)
\leq C\int_{\mathbb{R}^d\setminus\{0\}}(1\wedge|z|^2)\nu(dz)<\infty.
$$

\textbf{Step 4:}
From \eqref{bv-b}, \eqref{cv-c} and \eqref{gv-g}, the three
characteristics of $M^\varepsilon_{t}$ converge to those
of a symmetric $\alpha$-stable process $A_{t}$ with Lévy measure
defined by \eqref{A-mea}.  
Following the argument in \cite[p.1094, Theorem 1]{BF},  from \cite[p.459, Theorem 1.18]{JS} we have
$$
(M^\varepsilon_{t})_{t\in [0,T]} \xrightarrow{\mathcal D} (A_{t})_{t\in [0,T]}
\quad\text{in}\quad D([0,T];\mathbb R^d), \quad M^\varepsilon_{0}=0,
$$
by \eqref{a-m}, \eqref{rt}, and Slutsky's theorem, for
any finite collection 
$0\leq t_1<\cdots<t_k\leq T$, 
\[
(A_{t_1}^\varepsilon,\ldots,A_{t_k}^\varepsilon)
\xrightarrow{\mathcal D}
(-A_{t_1},\ldots,-A_{t_k})
\overset{\mathcal D}{=}
(A_{t_1},\ldots,A_{t_k}),
\]
we used  the symmetry of  $A_{t}$ in the last equality. Since $t_1,\ldots,t_k$ and $T>0$ are arbitrary,
we obtain
\[
(A^\varepsilon_t)_{t\geq0}\xrightarrow{\varepsilon\to 0}(A_t)_{t\geq0}, \quad   \text{ in finite-dimensional distributions,}
\]
the  proof is complete.
\end{proof}


\appendix

\section{Some technical estimates}\label{tech}

\begin{lemma}\label{v-A}
	
	For any matrix
	$K\in\mathbb R^{d\times d}$ with $\|K\|\leq1$, for any $R>0$, $ z\in \mathbb R^{d}\setminus\{0\}$,
	we have
	\begin{equation}\label{z1-2-A}
		\int_{\mathbb R^{d}\setminus\{0\}}\Big(V_\eta(z+Kh)-V_\eta(z)-\langle\nabla V_\eta(z),Kh\rangle
		I_{\{|h|\leq R\}}\Big)\nu(dh)\leq C_{\alpha,d}|z|^{1-\alpha},
	\end{equation}
	where $V_\eta(z)$ is defined in \eqref{vz}, and the constant $C_{\alpha,d}=
	c_{\alpha,d}|\mathbb S^{d-1}|
	\left(	2^{\alpha-2}/(2-\alpha)+2^{\alpha}/(\alpha-1)
	\right)$.
\end{lemma}

\begin{proof}
We split the integral into
	\begin{align*}
		H_1&=\int_{\{0<|h|\leq |z|/2\}}\left( V_\eta(z+Kh)-V_\eta(z)
		-\langle\nabla V_\eta(z),Kh\rangle I_{\{|h|\leq R\}}\right) \nu(dh),\\
		H_2&=\int_{\{|h|>|z|/2\}}\left( V_\eta(z+Kh)-V_\eta(z)
		-\langle\nabla V_\eta(z),Kh\rangle	I_{\{|h|\leq R\}}\right) \nu(dh),
	\end{align*}
	indeed, for $H_1$, by Taylor's formula, 
	$$V_\eta(z+Kh)=V_\eta(z)+\langle \nabla V_\eta(z), Kh\rangle
	+\int_0^1 (1-\theta)
	(Kh)^\top \nabla^2 V_\eta(z+\theta Kh) Kh\, d\theta,$$
	by the symmetry of $\nu$, the first-order term vanishes, and  \eqref{grad},  we have
	\begin{align*}
		H_1&=\int_{\{0<|h|\leq |z|/2\}}
		\Big(V_\eta(z+Kh)-V_\eta(z)
		-\langle\nabla V_\eta(z),Kh\rangle	I_{\{|h|\leq R\}}	\Big)\nu(dh)	\\
		&=	\frac12	\int_{\{0<|h|\leq |z|/2\}}
		\Big(V_\eta(z+Kh)+V_\eta(z-Kh)-2V_\eta(z)\Big)\nu(dh)
		\\
		&=   \frac12	\int_{\{0<|h|\leq |z|/2\}}	\Big(	\int_0^1(1-\theta)	(Kh)^{\top}
		\Big(\nabla^2V_\eta(z+\theta Kh)	+\nabla^2V_\eta(z-\theta Kh)\Big)Kh\,d\theta	\Big)\nu(dh)\\
		&\leq	\int_{\{0<|h|\leq |z|/2\}}\frac{|Kh|^2}{|z|}\nu(dh)\leq \frac{c_{\alpha,d}|\s^{d-1} |}{|z|}\int_0^{|z|/2}r^2 r^{-d-\alpha}r^{d-1}dr=	\frac{2^{\alpha-2}}{2-\alpha} c_{\alpha,d}|\s^{d-1}| |z|^{1-\alpha},
	\end{align*}
 in the first inequality, we used $\|K\|\leq1$, which gives
	$|Kh|\leq |h|$, and hence
	$$
	|z|/2  \leq |z|-|Kh|\leq   |z|-|\theta Kh| \leq|z\pm\theta Kh|,
	\qquad 0\leq\theta\leq1,
	$$
	and 
	$\|\nabla^2V_\eta(z\pm\theta Kh)\|\leq\frac{1}{|z\pm\theta Kh|}\leq\frac{2}{|z|}.$
	
	For $H_2$, by \eqref{grad} and global 1-Lipschitz continuity of $V_\eta$ in \eqref{v-lip}, and $\|K\|\leq1$, we obtain
	\begin{align*}
		H_2&=\int_{\{|h|>|z|/2\}}
		\Big(
		V_\eta(z+Kh)-V_\eta(z)
		-\langle\nabla V_\eta(z),Kh\rangle I_{\{|h|\leq R\}}
		\Big)\nu(dh)\\
		&\leq	2\int_{\{|h|>|z|/2\}}|Kh|\nu(dh)
		\leq	2c_{\alpha,d}|\s^{d-1} |\int_{|z|/2}^{\infty}r r^{-d-\alpha}r^{d-1}dr\leqslant	\frac{2^{\alpha}}{\alpha-1}c_{\alpha,d}|\s^{d-1}| |z|^{1-\alpha}.
	\end{align*}
	
	Combining the above estimates we yield
	\begin{equation*}
		\int_{\mathbb R^{d}\setminus\{0\}}
		\Big(	V_\eta(z+Kh)-V_\eta(z)
		-\langle\nabla V_\eta(z),Kh\rangle
		I_{\{|h|\leq R\}}\Big)\nu(dh)\leq C_{\alpha,d}|z|^{1-\alpha},
	\end{equation*}
	where $C_{\alpha,d}=
	c_{\alpha,d}|\mathbb S^{d-1}|
	\left(	2^{\alpha-2}/(2-\alpha)+2^{\alpha}/(\alpha-1)
	\right)$, 
	the proof is complete.
\end{proof}

Next we establish the moment estimate for the rescaled process
$Y_t^\varepsilon$, which has the same law as $Y_{t/\varepsilon}$.
Recall that $L_t$ denotes the isotropic $\alpha$-stable process in \eqref{multca}, then the Poisson random measures can be represented as \cite{DA},
$$ N(t,B)=\sum_{s\leq t}1_{B}(L_{s}-L_{s-}),\quad \forall B\in \mathcal{B}(\mathbb{R}^{d }\setminus\{0\}),$$
by L\'{e}vy-It\^{o} decomposition and symmetry of $\nu(dz)$, we have for any $r>0$
\begin{equation*}
	 L_{t}=\int_{|z|\leq r}z \tilde{N}(t,dz)+\int_{|z|> r}zN(t,dz).
\end{equation*}

Define
\[
L_t^\varepsilon:=\varepsilon^{1/\alpha}L_{t/\varepsilon},
\qquad t\ge0,
\]
by the self-similarity of the isotropic $\alpha$-stable
process, $L^\varepsilon_t$ is again an isotropic
$\alpha$-stable process and
\[
(L_t^\varepsilon)_{t\ge0}\overset{\mathcal D}{=}(L_t)_{t\ge0},
\]
 the time-rescaled process
$(Y_{t/\varepsilon})_{t\ge0}$ has the same law as the
unique strong solution $(Y_t^\varepsilon)_{t\ge0}$ of
\begin{equation}\label{3.2}
	dY_t^\varepsilon
	=\frac1\varepsilon f(Y_t^\varepsilon)\,dt	+\frac1{\varepsilon^{1/\alpha}}
	\delta(Y_{t-}^\varepsilon)\,dL_t,	\qquad
	Y_0^\varepsilon=y,
\end{equation}
then \eqref{3.2}
can equivalently be written as
\begin{align}\label{3.2-1}
	dY_t^\varepsilon
	=\frac1\varepsilon f(Y_t^\varepsilon)\,dt
	+\frac1{\varepsilon^{1/\alpha}}
	\int_{|z|\le\varepsilon^{1/\alpha}}
	\delta(Y_{t-}^\varepsilon)z\,\tilde N(dt,dz)
	+\frac1{\varepsilon^{1/\alpha}}
	\int_{|z|>\varepsilon^{1/\alpha}}
	\delta(Y_{t-}^\varepsilon)z\,N(dt,dz).
\end{align}

\begin{lemma}\label{T32}
	Let Assumptions $\ref{dis}$--$\ref{gro-bou}$ hold.
For any initial data $Y_0=y\in\mathbb{R}^{d}$, the time-rescaled process
$(Y_{t/\varepsilon})_{t\geq0}$ has the same law as the unique strong
solution $(Y_t^\varepsilon)_{t\geq0}$ of \eqref{3.2} with
$Y_0^\varepsilon=y$.
Moreover, for any $m\in[1,\alpha)$, there exists
	$C_m>0$ 
	such that
	\begin{equation}\label{3.6}
		\sup_{\varepsilon\in(0,1)}
		\sup_{t\geq0}\mathbb{E}|Y_t^\varepsilon|^m
		\leq C_m(1+|y|^m).
	\end{equation}
\end{lemma}
\begin{proof}
	Under Assumption $\ref{lip}$,  Assumption $\ref{gro-bou}$ on $f$ and $\delta$,	we derive the well-posedness of \eqref{3.2} by \cite[Theorem 6.2.3, Theorem 6.2.9, Theorem 6.2.11]{DA}. The proof of \eqref{3.6} refers to \cite[Lemma A.1]{SXX}.	
	
	From \eqref{3.2-1},
	due to the fact that $m<\alpha<2$,
	 we define   
$
				U(y)=(|y|^{2}+1)^{\frac{m}{2}},
$
	\begin{equation}\label{du}
			 |DU(y)|=\left|  \frac{my}{(|y|^{2}+1)^{1-\frac{m}{2}}}\right|\leq C_{m} |y|^{m-1},
	\end{equation}
	\begin{equation}\label{d2u}
			|D^{2}U(y)|=\left|   \frac{mI_{d\times d }}{(|y|^{2}+1)^{1-\frac{m}{2}}}-\frac{m(m-2)y\otimes y}{(|y|^{2}+1)^{2-\frac{m}{2}}} \right|\leq \frac{C_{m}}{(|y|^{2}+1)^{1-\frac{m}{2}}}\leq C_{m},
	\end{equation}
	applying It\^{o}'s formula and taking expectation on both sides, with $\mathbb{E}\tilde{N}(ds,dz)=0$ we derive,
	\begin{align*}
	\frac{d\mathbb{E}U(Y_{t}^{\varepsilon})}{dt}&=\mathbb{E} \frac{1}{\varepsilon}\langle  f(Y_{t}^{\varepsilon}),DU(Y_{t}^{\varepsilon})\rangle +\mathbb{E} \int_{|z|> \varepsilon^{\frac{1}{\alpha}}}\left( U(Y_{t-}^{\varepsilon}+\varepsilon^{-\frac{1}{\alpha}}\delta(Y_{t-}^\varepsilon) z)-U(Y_{t-}^{\varepsilon}) \right) \nu(dz) \\
		&\quad+\mathbb{E} \int_{|z|\leq \varepsilon^{\frac{1}{\alpha}}}\left( U(Y_{t-}^{\varepsilon}+\varepsilon^{-\frac{1}{\alpha}}\delta(Y_{t-}^\varepsilon) z)-U(Y_{t-}^{\varepsilon})-\langle  DU(Y_{t-}^{\varepsilon}), \varepsilon^{-\frac{1}{\alpha}}\delta(Y_{t-}^\varepsilon) z\rangle  \right) \nu(dz)\\
		&=I_{1}+I_{2}+I_{3},
	\end{align*}
for $I_{1}$, by  \eqref{2.2} in Assumption \ref{dis}, and \eqref{du}, there exists $ c_{m,\lambda}>0$ such that
\begin{align*}
I_{1}&=\mathbb{E} \frac{1}{\varepsilon}\langle  f(Y_{t}^{\varepsilon}),DU(Y_{t}^{\varepsilon})\rangle \\
& \leq\frac{1}{\varepsilon} \mathbb{E}  \frac{\langle  f(Y_{t}^{\varepsilon})-f(0),mY_{t}^{\varepsilon}\rangle  +\langle  f(0),mY_{t}^{\varepsilon}\rangle }{(|Y_{t}^{\varepsilon}|^{2}+1)^{1-\frac{m}{2}}}\\
&\leq  \frac{1}{\varepsilon}\mathbb{E} \frac{C_{m}|Y_{t}^{\varepsilon}|-m\lambda|Y_{t}^{\varepsilon}|^{2}}{(|Y_{t}^{\varepsilon}|^{2}+1)^{1-\frac{m}{2}}}
\leq\frac{C_{m}}{\varepsilon}
-\frac{m\lambda}{\varepsilon}
\mathbb{E}\frac{|Y_{t}^{\varepsilon}|^{2}}
{(|Y_{t}^{\varepsilon}|^{2}+1)^{1-\frac{m}{2}}}\leq \frac{C_{m}}{\varepsilon}-\frac{c_{m,\lambda}\mathbb{E}U(Y_{t}^{\varepsilon})}{\varepsilon},  
\end{align*}
	in addition, taking $y=\varepsilon^{-\frac{1}{\alpha}}z$,  
thus for $I_{2}$, by  the mean value theorem, Assumption \ref{gro-bou} of $ \delta(y)$ and \eqref{du}, there exists $ c_{m}>0$ such that
\begin{align*}
		&I_{2}=\frac{1}{\varepsilon}\mathbb{E} \int_{|y|> 1}\big( U(Y_{t-}^{\varepsilon}+\de(Y_{t-}^\varepsilon) y)-U(Y_{t-}^{\varepsilon}) \big) \nu(dy) \\
		&\quad\leq  \frac{C_{m}}{\varepsilon}\mathbb{E}\int_{|y|> 1}\left(|Y_{t}^{\varepsilon}|^{m-1}|y|+|y|^{m} \right) \nu(dy) \leq\frac{C_{m}}{\varepsilon}+\frac{c_{m}\mathbb{E}U(Y_{t}^{\varepsilon})}{\varepsilon},  
\end{align*}
then by Taylor's formula, and \eqref{d2u},
\begin{equation*}
I_{3}=\frac{1}{\varepsilon}\mathbb{E} \int_{|y|\leq 1}\big( U(Y_{t-}^{\varepsilon}+\de(Y_{t-}^\varepsilon) y)-U(Y_{t-}^{\varepsilon})-\langle  DU(Y_{t-}^{\varepsilon}),\de(Y_{t-}^\varepsilon) y\rangle  \big) \nu(dy)\leq  \frac{C_{m}}{\varepsilon},  
\end{equation*}
combining the above estimates, we obtain
\begin{equation*}
	\frac{d\mathbb{E}U(Y_{t}^{\varepsilon})}{dt}
	\leq\frac{C_{m}}{\varepsilon}
	-\frac{(c_{m,\lambda}-c_m)\mathbb{E}U(Y_{t}^{\varepsilon})}{\varepsilon},
\end{equation*}
choosing $\lambda$ in \eqref{2.2} sufficiently large such that
$c_{m,\lambda}-c_m>0$, and denote $c_{m,\lambda}-c_m$ by $c_m$, we derive
\begin{equation*}
	\frac{d\mathbb{E}U(Y_{t}^{\varepsilon})}{dt}
	\leq\frac{C_{m}}{\varepsilon}
	-\frac{c_{m}\mathbb{E}U(Y_{t}^{\varepsilon})}{\varepsilon}.
\end{equation*}
so that by Gronwall's inequality we have 
\begin{align*}
	\mathbb{E}U(Y_{t}^{\varepsilon})\leq e^{-c_{m}\frac{t}{\varepsilon}}(|y|^{2}+1)^{\frac{m}{2}}+\frac{C_{m}}{\varepsilon} \int^{t}_{0}e^{-\frac{c_{m}}{\varepsilon}(t-s)}ds,
\end{align*}
which means
$$\mathbb{E}(|Y_{t}^{\varepsilon}|^{2}+1)^{\frac{m}{2}}\leq e^{-c_{m}\frac{t}{\varepsilon}}(|y|^{2}+1)^{\frac{m}{2}}+C_{m}(1-e^{-c_{m}\frac{t}{\varepsilon}}),$$
so that there exists $C_{m}>0$ s.t.
\begin{equation*}
	\sup\limits_{\varepsilon\in (0,1)}\sup\limits_{t\geq0} \mathbb{E}\left(|Y_{t}^{\varepsilon}|^{m}\right)\leq C_{m}(1+|y|^{m}),
\end{equation*}
the proof is complete.
\end{proof}

\section*{Acknowledgements}
The author is very grateful to the invaluable guidance from Professor Xin Chen in Shanghai Jiao Tong University as the doctoral supervisor, particularly for  the crucial references and useful communications. Special acknowledgement is extended to Dr. Qiu-Chen Yang for helpful discussions.

\end{document}